%% file: Pricing_anna-ITA.tex
\documentclass[10pt,conference]{IEEEtran}
\usepackage{amsthm}

\usepackage[T1]{fontenc}
\usepackage[utf8]{inputenc}
\usepackage[english]{babel}
\usepackage[cmex10]{amsmath}
\usepackage{amssymb}
\usepackage{enumerate}
\usepackage{paralist}
\usepackage{tikz}
\usepackage{pgfplots}
\usepackage{xcolor}
\usepackage{subcaption}

\usepgfplotslibrary{patchplots}

\usetikzlibrary{shapes.geometric,calc}

 \newtheorem{theorem}{Theorem}[section]
 
 \newtheorem{lemma}[theorem]{Lemma}

 \newtheorem{definition}[theorem]{Definition}
 
\usepackage{amssymb}

\usepackage{cite}

\ifCLASSINFOpdf
   \usepackage{graphicx}
\else
\fi

\usepackage[cmex10]{amsmath}
\usepackage{algorithmic}
\usepackage{array}
\usepackage{url}

\begin{document}

\title{Continuous-time Marginal Pricing of Power Trajectories in Power Systems}

\author{Anna Scaglione,~\IEEEmembership{Fellow,~IEEE} 
\thanks{A. Scaglione is with the School of Electrical, Computer and Energy Engineering, Arizona State University, Tempe, AZ 85287 (e-mail: anna.scaglione@asu.edu)}
}


\input{macros}

\maketitle

\begin{abstract}
In this paper we formulate of the Economic Dispatch (ED) problem in Power Systems in continuous time and include in it ramping constraints to derive an expression of the price that reflects some important inter-temporal constraints of the power units. 
The motivation for looking at this problem is the scarcity of ramping resources and their increasing importance motivated by the variability of power resources, particularly due to the addition of solar power, which exacerbates the need of fast ramping units in the early morning and early evening hours. We show that the solution for the marginal price can be found through Euler-Lagrange equations  and we argue that this price signal better reflects the market value of power in the presence of significant ramps in net-load. 
\end{abstract}


\IEEEpeerreviewmaketitle

\section{Introduction}

In today's  electricity markets, energy and ancillary services serve two different purposes. A particularly important stage is the day-ahead market in which generating units only define an hourly cost for their average energy generation schedule during that unit of time, i.e. a function $C_k(\bar x,h)$, where the hourly energy signal in an horizon ${\cal T}=[t_1,t_2]$, with $t_2=,t_1+HT$ is: 
$${\bar x}_k[h]=\int_{t_1+hT}^{(h+1)T} x_k(t)dt, ~~~h=0,\ldots,H-1$$
First, the best portfolio of units is decided by solving the mixed integer linear program that corresponds to the Unit Commitment (UC).  The price is then chosen based on the portfolio of units selected, setting the optimum schedule and is  interpreted as an energy price for that particular period and market settlement.
 
Demand of electricity does not come in chucks of energy and what is demanded from the market is a continuous time power trajectory. Today's operations catch up in real time in shorter term market settlements and engaging ancillary services reserves. 
While it is reasonable for the generation schedules to be merely in the right ballpark and adjust to the actual demand in real time it is important to consider that, using this approximation, one may be throwing away information that is at hand about the inter-hour behavior and, particularly, load ramps that may be difficult to meet with the units committed with the standard UC formulation. 
Recently this issue has come to the forefront due to increased variability of the net-demand, particularly at early and late day hours, that are a direct consequence of the increased penetration of solar power. 
This phenomenon is often illustrated by the popular {\it duck chart} in Fig. \ref{fig.duckchart}.
\begin{figure}[htbp]
\begin{center}
\includegraphics[width=7cm]{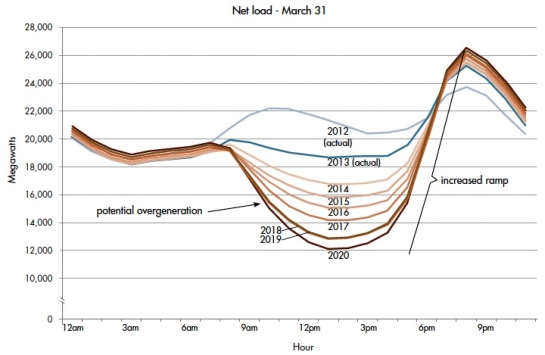}
\caption{The forecast of average California daily load.}
\label{fig.duckchart}
\end{center}
\end{figure}

Metrics for measuring the largest variation range of uncertainty that the system can accommodate were proposed in \cite{ma2009}, while the authors in \cite{lefton2006cost,kumar2012power} proposed methods to quantify the additional cost of generating units incur due to ramping. 
At the same time, there have been several proposal to address the issue of ramping shortage which promote better coordination between energy markets, security constrained Unit Commitment and ancillary service markets and higher incentives for ramping services. Oftentimes the framework proposed is tied to the uncertainty of renewables and the Security Constrained UC (see e.g. \cite{wang2008security,wu2015thermal,bouffard2008stochastic,thatte2014robust}). Cases in which the authors more explicitly tied the ramping cost to modeling the ramping needs in the UC are, for instance, \cite{troy2012unit,wu2013hourly,wang2013flexiramp}. On the market side Midcontinent ISO (MISO) \cite{navid2011,navid2012} and the California ISO (CAISO)  \cite{xu2012flexible,abdul2012enhanced} introduced flexible ramping products in their portfolio of resources.

Accounting for the predictable trends and incentivizing the market properly could lead to decisions that reduce the volume of power that is scheduled in real time.  This is, in a nutshell the motivation behind our previous work in \cite{parvania2015UC}, \cite{parvaniaHICSS2016} \cite{KariPSCC2016} that essentially examined how to convert the intractable variational problem of committing generators trajectories, i.e. solving the Unit Commitment (UC) problem, in continuous time, into a tractable problem, resorting to an expansion of the demand and generation trajectories in a signal space spanned by polynomial splines. 
We brought the notion of continuous time UC a step further in \cite{parvaniaHICSS2016} where, 
assuming we are going to deal with continuous and differentiable trajectories, we defined a unit cost that is a function of both ${x}_k(t)$ and $\dot{x}_k(t)$ $k=1,\ldots,K$, to allow generators to bid also on their ramping capacity, and stimulate market competition in offering ramping services. More recently we proposed to fully generalize export bids 
to possibly account for storage capacity available \cite{KariPSCC2016}.

The subject of this paper is describing simple instances of the variational problem that defines the price of electricity without involving the integer variables of the UC, as is customary in computing the price of energy. We focus on the economic dispatch problem with the goal of calculating the  trajectory of the marginal price per unit of power and time, considering a market that maximizes social surplus in the presence of inflexible demand. 

The paper is organized as follows.
We start in the following section with the classic economic dispatch problem, but look the the optimization over the entire time horizon and include in the calculation ramping constraints.   

\section{The Continuous Time Economic Dispatch Problem and Pricing of Power}
We consider energy market setting in which a set of $K$ generating units compete to sell continuous-time generation trajectory $x_k(t)$ for $k=1,\ldots,K$. The input to the problem is the forecasted load profile $y(t)$ over the day-ahead operating horizon ${\mathcal T}=[t_1,t_2]$ at minimum cost, subject to the constraints of the generating units that include their capacity and ramping limits. 
Our interest in this paper is the derivation of the marginal price of electricity by first focusing on the economic dispatch formulation that relaxes the power flow equations and the line flow constraints. We assume that the UC problem has been solved, for instance, following the setup we proposed in \cite{parvania2015UC}. Hence, the $K$ units considered here were the one that were committed by solving the UC problem. 
 

We can assume that generating unit $k$ expresses a cost function $C_k(x;t)$ such that the incremental cost for generating the energy $x(t)dt$ is $C_k(x;t)dt$. This is perfectly compatible with the current practice of giving an hourly or sub-hourly energy cost, since 
in that case $C_k(x;t)$ is linear (or piece-wise) linear with respect to $x$, over the horizon ${\mathcal T}$, with exactly the same rate. 
The formulation to calculate the market clearing price in continuous time is: 
\begin{align} \label{Var1}
\min_{x_1(t),..,x_K(t)}& ~~~ \sum_{k=1}^K \int_{t_1}^{t_2} \!\!C_k(x_k,t)dt \\ \label{Var2}
\mbox{s.t.}~~&\sum_{k=1}^K x_k(t)=y(t)  ~~ (\lambda(t)) ~~ \forall t \in {\mathcal T} \\ \nonumber
&~~ \forall t \in {\mathcal T},~~k=1,\ldots,K \\ \label{Var3}
&x_k(t) -\overline{x}\leq 0 ~~ (\overline{\mu}_k(t))\\\label{Var4}
&\underline{x}_k -x_k(t) \leq 0 ~~ (\underline{\mu}_k(t)) \\ \label{Var5}
&{\dot x}_k(t) -\overline{x}_k' \leq 0 ~~ (\overline{\gamma}_k(t))  \\ \label{Var6}
&\underline{x}_k' -\dot{x}_k(t) \leq 0 ~~ (\underline{\gamma}_k(t)) 
\end{align}
where $\underline{x}_k$, $\overline{x}_k$, $\underline{x}_k' $, $\overline{x}_k' $ are respectively the upper and lower bounds of the $k$th generation and ramping trajectory, and $\lambda(t)$, $\underline{\mu}_k(t)$, $\overline{\mu}_k(t)$, $\underline{\gamma}_k(t)$, $(\overline{\gamma}_k(t)$ are respectively the Lagrange multipliers associated with the constraints \eqref{Var2}-\eqref{Var6}).
The objective in \eqref{Var1} is to minimize the total continuous-time cost of generating units over the scheduling horizon ${\mathcal T}$, subject to the continuous-time load-generation balance constraint \eqref{Var2}, and the continuous-time generation capacity and ramping constraints in \eqref{Var3}-\eqref{Var6}. 
Here, we assume there is no flexibility in $y(t)$ and its trajectory is certain.

The optimization problem in \eqref{Var1}-\eqref{Var6} is a constrained variational problem, and its solution would determine the values of the Lagrange multipliers over the scheduling horizon ${\mathcal T}$, that correspond to the optimal generation trajectories $\mathbf{x}^*(t)$ for the units that are committed. 
The Lagrangian function of the variational problem \eqref{Var1}-\eqref{Var6} is:
 \begin{align}\label{eq.lagrangian}
{\cal L}=  
& \sum_{k=1}^K \int_{t_1}^{t_2}f_k(x_k,{\dot x}_k,t)dt\\
 f_k(x_k,{\dot x}_k,t)&=C_k(x_k(t),t)+\left(\frac{y(t)}{K}-x_k(t)\right)\lambda(t) \nonumber\\
+& \label{def.f}
\overline{\mu}_k(t)(x_k(t) -\overline{x}_k)+
\underline{\mu}_k(t)(\underline{x}_k -x_k(t)) \\
+& 
\overline{\gamma}_k(t)({\dot x}_k(t) -\overline{x}_k') +
\underline{\gamma}_k(t)(\underline{x}_k' -\dot{x}_k(t))\label{def.fk}
\end{align}
In the following, what we refer to as the {\it continuous-time Lagrange multiplier} $\lambda(t)$ function. In fact, each instantaneous balance constraint needs to be weighted by an infinitesimal quantity $\lambda(t)dt$,
just like the cost of supplying $x_k(t)$ over the interval $[t,t+dt)$ grows infinitesimally by $C_k(x_k(t))dt$. 
Assuming $y(t)$ and $C_k(x,t)$ are continuous and continuously differentiable function of time, the trajectories $x_k(t)$ and Lagrange multipliers are going to be also continuous and continuously differentiable. 
The formulation above is a special case of the classical {\it isoperimetric problem} from Physics (see e.g. \cite{variational}). The key result is that the trajectories $x_k(t)$, $k=1,\ldots,K$ that lead to the minimum ${\cal L}^*\leq {\cal L}$ solutions can be found solving the so called Euler-Lagrange differential equations:
\begin{align}\label{eq.KKT1}
& \frac{\partial  f_k(x_k^*,\dot{x}_k^*,t)}{\partial x_k}-\frac{d}{dt}\frac{\partial f(x_k^*,\dot{x}_k^*,t)}{\partial {\dot x}_k}=0,~ k=1,..,K
\end{align}
 Because there are no specific constraints on the values of the trajectories at least at one of the boundaries of ${\mathcal T}$, there is one additional condition to consider:
\begin{align}\label{eq.dgamma}
\frac{\partial f(x^*_k(t_2),{\dot x}^*_k(t_2),t_2)}{\partial {\dot x}_k}&= \frac{d\overline{\gamma}^*_k(t_2)}{dt}-\frac{d\underline{\gamma}^*_k(t_2)}{dt}=0\\
& ~~k=1,\ldots,K\nonumber
\end{align}

In addition to \eqref{eq.KKT1} the complete Karush-Kuhn-Tucker (KKT) conditions on the optimum trajectories include the following equations:
\begin{align} 
&y(t)-\sum_{k=1}^K x^*_k(t)=0  \label{eq.KKT2} \\
&x_k^*(t)-\overline{x}_k\leq 0,~~\underline{x}_k-x^*_k(t)\leq 0,\label{eq.KKT3} \\
&{\dot x}^*_k(t)-\overline{x}_k'\leq 0,~~\underline{x}_k'-{\dot x}^*_k(t)\leq 0, \\
&\overline{\mu}_k(t)( x_k^*(t)-\overline{x}_k)=0,~~ 
\underline{\mu}_k(t)(\underline{x}_k-x_k^*(t))=0, \label{eq.cse1} \\
&\overline{\gamma}_k(t)({\dot x}_k^*(t)-\overline{x}_k')=0, ~~
\underline{\gamma}_k(t)(\underline{x}'_k-{\dot x}^*_k(t))=0, \label{eq.cse2} \\
&\overline{\mu}_k(t),\underline{\mu}_k(t),\overline{\gamma}_k(t), \underline{\gamma}_k(t) \geq 0. \label{eq.df}\\
\forall& t\in {\mathcal T},~~~k=1,\ldots,K\nonumber
\end{align}
All of these equations are decoupled with the exception of the balance constraint.
Based on the definition of $f(x_k,{\dot x}_k,t)$ in \eqref{def.f} we have:
\begin{align}
&\frac{\partial  f(x_k,{\dot x}_k,t)}{\partial x_k}=\frac{\partial  C_k(x_k,t)}{\partial x_k}-\lambda(t)+\overline{\mu}_k(t)-\underline{\mu}_k(t)\label{eq.parx}\\
&\frac{d}{dt}\frac{\partial  f(x_k,{\dot x}_k,t)}{\partial {\dot x}_k}=\frac{d\overline{\gamma}_k(t)}{dt}-\frac{d\underline{\gamma}_k(t)}{dt},\label{eq.pardx}\\~~
&\forall t\in {\mathcal T},~k=1,\ldots,K\nonumber
\end{align}
Therefore, from the \eqref{eq.KKT2} we can conclude a number of things. 
\begin{enumerate}
\item 
First, 
we can derive the following relationship between the Lagrange multiplier function, the marginal cost and the other Lagrange multipliers functions:
\begin{align}\nonumber
&\!\!\!\!\!\!\!\!\lambda^*(t)=\frac{\partial  C_k(x^*_k,t)}{\partial x_k}+
\overline{\mu}^*_k(t)-\underline{\mu}^*_k(t)-\frac{d\overline{\gamma}^*_k(t)}{dt}+\frac{d\underline{\gamma}^*_k(t)}{dt}\\
\forall& t\in {\mathcal T},~~~k=1,\ldots,K \label{eq.mag.price}
\end{align}

\item
For each unit $k$ that is operating strictly within its generation and ramping capacity limits, the complementarity slackness condition \eqref{eq.cse1}-\eqref{eq.cse2} imply that  the multipliers 
$\overline{\mu}_{k^*}(t)=\underline{\mu}_{k^*}(t)=0$ and/or $\overline{\gamma}_{k^*}(t)=\underline{\gamma}_{k^*}(t)=0$.

\item It is important to note that because, due to the same conditions, the eigenvalues $\overline{\gamma}_{k^*}(t),\underline{\gamma}_{k^*}(t)$ are non negative \eqref{eq.df}, then
zero is also the infimum of the function. Because the function is continuous with continuous derivative in time at the infimum
will also have to be a minimum and have zero derivative. This means that when the ramping constraint is not tight, the following is also true:
\begin{equation}\label{eq.gamma}
\frac{d\overline{\gamma}^*_k(t)}{dt}=0,~~\frac{d\underline{\gamma}^*_k(t)}{dt}=0
\end{equation}
and, thus 
\begin{equation}\label{eq.dtpf}
\frac{d}{dt}\frac{\partial  f(x_k,{\dot x}_k,t)}{\partial {\dot x}_k}=0
\end{equation}
\item In our problem, since we assume that the demand is inflexible, to  always have a feasible solution we have to assume that there always exist an extra unit that can be deployed to meet excess demand that arises. 
The so called {\it marginal unit} is the unit $k^*(t)$ for which none of the inequality constraints is tight at time $t$ and so $\overline{\mu}_{k^*}(t)=\underline{\mu}_{k^*}(t)=0$ and/or $\overline{\gamma}_{k^*}(t)=\underline{\gamma}_{k^*}(t)=0$ and also \eqref{eq.dtpf} is valid. 
In this case, assuming that there is always a unit whose constraints are not tight and denoting by $k^*(t)$ the index of the marginal unit at time $t$ 
the solution of \eqref{eq.KKT1} is: 
\begin{align}\label{eq.marginal}
&\lambda^*(t)=\frac{\partial  C_{k}(x^*_{k},t)}{\partial x_k},~~~~k={k^*}(t),~~\forall t\in {\mathcal T}.
\end{align}
Note that, if the increment in demand has an arbitrary temporal shape, it may be true that only the ramping constraint is tight though the problem is still feasible, in which case the price will include either 
$\frac{d\underline{\gamma}^*_k(t)}{dt}$ or $\frac{d\overline{\gamma}^*_k(t)}{dt}$, i.e. there would be a shadow price associated to ramping that increases the price compared to the marginal cost.
\end{enumerate}
It is well known that, absent a load export bid, the marginal price of electricity and the marginal cost of the so called marginal unit are the same, since the social surplus is maximized by having the marginal generator increase its production by one unit if the demand grows by one unit. The same remains true in the continuous time formulation. 
We can prove this formally,  to state the following theorem:
\begin{theorem}
With a continuous time import bid and inflexible demand the surplus maximizing price for an additional unit of load per unit of time is the Lagrange multiplier $\lambda^*(t)$ at the optimum schedule, i.e. \eqref{eq.marginal}.
\end{theorem}
\begin{proof}
Let $g_k(x_k,{\dot x}_k,t)$ be:
\begin{equation}
g_k(x_k,{\dot x}_k,t)=f_k(x_k,{\dot x}_k,t)-\frac{y(t)}{K},
\end{equation}
so that we can rewrite the Lagrangian as follows:
\begin{equation}
{\cal L}=\int_{t_1}^{t_2} \left(\lambda(t)y(t)+\sum_{k=1}^Kg_k(x_k,{\dot x}_k,t)\right)dt
\end{equation}
Suppose we increase the load demand by a unit of power $\epsilon$ at all times in ${\cal T}$.
That is, the load is:
\begin{align}
Y(t)&=y(t)+\epsilon,~~~t\in (t_1,t_2)\\
Y(t_1)&=y(t_1),~~Y(t_2)=y(t_2). 
\end{align}
The last equations are necessary for the Euler-Lagrange equations to continue to hold. 
To evaluate the price this costs, we want to establish at what rate the Lagrangian changes compared to the value ${\cal L}^*$, as our demand of power increases. Let us refer to the new value of the Lagrangian as ${\cal L}^*(\epsilon)$.
Note that at the extremum only the marginal unit can adjust its production and ramp up to meet the increased load, all the remaining unit schedules for $k\neq k^*$ remain unchanged. 
Furthermore, since the change in load is a constant lift, the time derivative of the optimum schedule for the marginal unit does not change: 
\begin{equation}
X_{k^*}(t)=x_{k^*}^*(t)+\epsilon~\rightarrow~ {\dot X}_{k^*}(t)={\dot x}_{k^*}^*(t).
\end{equation}
These considerations lead to the following:
\begin{align}
{\cal L}^*(\epsilon)-{\cal L}^*=
\int_{t_1}^{t_2} \!\!\!\left(\lambda(t)\epsilon+\frac{\partial g_k(x^*_{k^*},{\dot x}^*_{k^*},t)}{\partial x_{k^*}^*}\epsilon +{\cal O}(\epsilon^2)\right)dt
\end{align}
where the partial derivative with respect to ${\dot x}^*_{k^*}$ is absent since we have no variation in the ramp. 
Considering that for the marginal unit:
\begin{align}
\frac{\partial g_k(x^*_{k^*},{\dot x}^*_{k^*},t)}{\partial x_{k^*}^*}&=
\frac{\partial f_k(x^*_{k^*},{\dot x}^*_{k^*},t)}{\partial x_{k^*}^*}\\&=\frac{\partial C_k(x^*_{k^*},t)}{\partial x_{k^*}^*}-\lambda^*(t)=0,
\end{align}
we can conclude that:
\begin{align}
\lim_{\epsilon \rightarrow 0} 
\frac{{\cal L}^*(\epsilon)-{\cal L}^*}{\epsilon}=
\int_{t_1}^{t_2} \lambda(t)dt
\end{align}
Hence, $\lambda(t)$ represents the market price per unit of power and per unit of time. 
\end{proof}
Note that, had we perturbed the load in a time varying fashion, i.e.:
$$Y(t)=y(t)+\epsilon\eta(t),$$ 
with $\eta(t_1)=\eta(t_2)=0$, but we never exceed the ramping constraint of the marginal unit, 
the result would have been:
\begin{align}
\lim_{\epsilon \rightarrow 0} 
\frac{{\cal L}^*(\epsilon)-{\cal L}^*}{\epsilon}=
\int_{t_1}^{t_2} \lambda(t)\eta(t)dt
\end{align}
which leads to the same conclusion. There could be a feasible $\eta(t)$ that
make the ramping constraint tight and, therefore, in that case the price would rise above 
the marginal cost and include, as we mentioned, either 
$\frac{d\underline{\gamma}^*_k(t)}{dt}$ or $\frac{d\overline{\gamma}^*_k(t)}{dt}$.

\section{Pricing with Generalized Bid Structures}
In an effort to better reflect constraints and costs of the generating units in our previous work we proposed to generalize the bid structure. Specifically, to incentivize suppliers to compete in offering ramping services, in \cite{parvaniaHICSS2016} we proposed to redefine export price bids so that the cost of generation at a certain time was function growing monotonically with respect to both power and the absolute value of the ramp. Later, to incentivize generators to rely on storage capacity in their generation portfolio, in \cite{KariPSCC2016} we generalized the notion further to make the export bid a function of energy, power and ramp. 

%
In this case, the bid may no longer be a simple export bid, but also include a price for purchasing power, as discussed in in \cite{KariPSCC2016}, that would allow the generator to buy at low prices and sell in later period when shortages raise the whole-sale price. Our previous work focused on the UC problem, where the presence of  resources that are more flexible, but constrained in  
generation or storage capacity, can better compete in the market and decrease the social 
cost of meeting highly variable demand. 

Naturally,  the generalized bid impacts the schedule and the price computed through the economic dispatch. Let us start from the first case, where the cost is function of both power and ramp. In this case:
\begin{theorem} 
When the generator costs depend on both power and ramp, i.e. $C_k(x_k,{\dot x}_k,t)$, the marginal price, i. e. the price for lifting uniformly during the horizon the power demand by one unit, is:
\begin{equation}
\lambda^*(t)=\frac{\partial  C_k(x^*_k,{\dot x}^*_k,t)}{\partial x_k}-\frac{d}{dt}\frac{\partial  C_k(x^*_k,{\dot x}^*_k,t)}{\partial {\dot x}_k},~~k=k^*.
\label{price.C(x,dotx)}
\end{equation}
\end{theorem}
\begin{proof}
The only thing that has changed compared to the previous case is that the objective is a function of both power and ramp, but Euler-Lagrange equations can still be directly applied. The KKT conditions and \eqref{eq.parx} remain the same, since the constraints have not changed, but a new term appears in \eqref{eq.pardx} due to the ramping constraint, which is not part of the traditional economic dispatch problem:
\begin{align}
&\frac{d}{dt}\frac{\partial  f(x_k,{\dot x}_k,t)}{\partial {\dot x}_k}=
\frac{d}{dt}\frac{\partial  C_k(x_k,{\dot x}_k,t)}{\partial {\dot x}_k}-\frac{d\overline{\gamma}_k(t)}{dt}+\frac{d\underline{\gamma}_k(t)}{dt},\label{eq.pardx-gen}\\~~
&\forall t\in {\mathcal T},~k=1,\ldots,K\nonumber.
\end{align}
Therefore, simply applying the Euler-Lagrange equations, and focusing on the marginal unit for which $\overline{\mu}_{k^*}(t)=\underline{\mu}_{k^*}(t)=0$ and/or $\overline{\gamma}_{k^*}(t)=\underline{\gamma}_{k^*}(t)=0$ and also \eqref{eq.dtpf} is valid the claim follows.
\end{proof}

To express the most general bid form, let:
\begin{equation}
z_k(t)=\int_{t_1}^t x_k(t)dt,~{\dot z}_k(t)=x_k(t),~{\ddot z}_k(t)={\dot x}_k(t).
\end{equation}
 The generators cost in this case is a function of the energy power and ramp, i.e.
 $C_k(z_k,{\dot z}_k,{\ddot z}_k,t)$. In this case the {\it isoperimetric problem}  needs to be generalized, due to the presence of the second derivative. 
The generalized formulation is now:
\begin{align} \label{Var1-gen}
\min_{\mathbf{z}(t)}& ~~~ \sum_{k=1}^K \int_{t_1}^{t_2} \!\!C_k(z_k,{\dot z}_k,{\ddot z}_k,t)dt \\ \label{Var2-gen}
\mbox{s.t.}~~&\sum_{k=1}^K z_k(t)=y(t)  ~~ (\lambda(t)) ~~ \forall t \in {\mathcal T} \\ 
&~~ \forall t \in {\mathcal T},~~k=1,\ldots,K\nonumber \\ \label{Var3-gen}
&{\dot z}_k(t) -\overline{x}_k\leq 0 ~~ (\overline{\mu}_k(t))\\\label{Var4-gen}
&\underline{x}_k -{\dot z}_k(t) \leq 0 ~~ (\underline{\mu}_k(t)) \\ \label{Var5-gen}
&{\ddot z}_k(t) -\overline{x}_k' \leq 0 ~~ (\overline{\gamma}_k(t))  \\ \label{Var6-gen}
&\underline{x}_k' -{\ddot z}_k(t) \leq 0 ~~ (\underline{\gamma}_k(t))\\\label{Var7-gen}
&z_k(t) -\overline{z}_k \leq 0 ~~ (\overline{\beta}_k(t))
\end{align}
where \eqref{Var7-gen} 
represent the additional energy constraint\footnote{Note that the energy is always positive as long as the lower bound on power $\underline{x}_k\geq 0$.}. 
Also in this case we can express the Lagrangian as follows:
 \begin{align}\label{eq.lagrangian-gen}
{\cal L}=  
& \sum_{k=1}^K \int_{t_1}^{t_2}\!\!\!f_k(z_k,{\dot z}_k,{\ddot z}_k,t)dt\\
f_k(z_k,{\dot z}_k,{\ddot z}_k,t)&=C_k(z_k,{\dot z}_k,{\ddot z}_k,t)+\left(\frac{y(t)}{K}-{\dot z}_k(t)\right)\lambda(t) \nonumber\\
+& 
\overline{\mu}_k(t)({\dot z}_k(t) -\overline{x})+
\underline{\mu}_k(t)(\underline{x}_k -{\dot z}_k(t)) \nonumber\\
+& 
\overline{\gamma}_k(t)({\ddot z}_k(t) -\overline{x}_k') +
\underline{\gamma}_k(t)(\underline{x}_k' -{\ddot z}_k(t))\nonumber\\
+& 
\overline{\beta}_k(t)(z_k(t) -\overline{z}_k) 
\label{def.f-gen}
\end{align}
Following the same principles that lead to the derivation of the Euler-Lagrange equation, we can prove the following:
\begin{lemma}\label{lem.gen-EuLag}
The trajectories that lead to the minimum ${\cal L}^*$ satisfy the following differential equations:
\begin{align}
& \frac{\partial  f_k(z_k^*,\dot{z}_k^*,\ddot{z}_k^*,t)}{\partial {\dot z}_k}-\frac{d}{dt}\frac{\partial f(z_k^*,\dot{z}_k^*,\ddot{z}_k^*,t)}{\partial {\ddot z}_k}\nonumber\\
&-\int_{t_2}^t\frac{\partial f(z_k^*,\dot{z}_k^*,\ddot{z}_k^*,\tau)}{\partial z_k}d\tau=0,~~~ k=1,...,K, t\in {\cal T}\label{eq.EuLag-gen}
\end{align}
in addition to the following KKT conditions:
 \begin{align} 
&y(t)-\sum_{k=1}^K {\dot z}^*_k(t)=0  \label{eq.KKT2-gen} \\
&{\dot z}_k^*(t)-\overline{x}_k\leq 0,~~\underline{x}_k-{\dot z}^*_k(t)\leq 0,\label{eq.KKT3-gen} \\
&{\ddot z}^*_k(t)-\overline{x}_k'\leq 0,~~\underline{x}_k'-{\ddot z}^*_k(t)\leq 0, \\
&z^*_k(t)-\overline{z}_k\leq 0,~~\underline{z}_k-z^*_k(t)\leq 0, \\
&\overline{\mu}_k(t)( {\dot z}_k^*(t)-\overline{x}_k)=0,~~ 
\underline{\mu}_k(t)(\underline{x}_k-{\dot z}_k^*(t))=0, \label{eq.cse1-gen} \\
&\overline{\gamma}_k(t)({\ddot z}_k^*(t)-\overline{x}_k')=0, ~~
\underline{\gamma}_k(t)(\underline{x}'_k-{\ddot z}^*_k(t))=0, \label{eq.cse2-gen} \\
&\overline{\beta}_k(t)(z_k^*(t)-\overline{z}_k)=0, ~~
\underline{\beta}_k(t)z^*_k(t)=0, \label{eq.cse3-gen} \\
&\overline{\mu}_k(t),\underline{\mu}_k(t),\overline{\gamma}_k(t), \underline{\gamma}_k(t), 
\overline{\beta}_k(t)
\geq 0. \label{eq.df-gen}\\
\forall& t\in {\mathcal T},~~~k=1,\ldots,K\nonumber
\end{align}
\end{lemma}
\begin{proof}
To prove \eqref{eq.EuLag-gen}, we shall follow the same steps as the proof of Euler-Lagrange equations. Their derivation is based on the fact that if we perturb the optimum point with an energy neutral trajectory as follows:
\begin{equation}
{\dot Z}_k(t)={\dot z}^*_k(t)+\epsilon_k {\dot \eta}_k(t), ~~k=1,\ldots,K
\end{equation}
the functional gradient of the Lagrangian needs to be zero when $\epsilon_k=0$, since
${\cal L}$ is minimized by $z^*_k(t)$, i.e. ${\cal L}^*\leq {\cal L}$. If we assume that $\eta_k(t_2)={\dot \eta}_k(t_1)={\dot \eta}_k(t_2)=0$ ($\eta_k(t_2)$ implies that the perturbation is energy neutral),  this means that the following integrals, no matter the choice of ${\dot \eta}_k(t)$, have to be zero:
\begin{align}
\left.
\frac{\partial {\cal L}}{\partial \epsilon_k} 
\right|_{\epsilon_k=0}\!\!\!\!\!\!
&=
\int_{t_1}^{t_2}\left(
\frac{\partial  f_k(z_k^*,\dot{z}_k^*,\ddot{z}_k^*,t)}{\partial z_k}\eta_k(t)
\right.\nonumber\\
\!\!\!\!\!\!&
+\left.
\frac{\partial f(z_k^*,\dot{z}_k^*,\ddot{z}_k^*,t)}{\partial {\dot z}_k}{\dot \eta}_k(t)
+\frac{\partial f(z_k^*,\dot{z}_k^*,\ddot{z}_k^*,t)}{\partial {\ddot z}_k}{\ddot \eta}_k(t)
\right)dt\nonumber\\
&=\left.
\left(
\int_{t_1}^t\frac{\partial f(z_k^*,\dot{z}_k^*,\ddot{z}_k^*,\tau)}{\partial {\dot z}_k}d\tau\right)\eta_k(t)
\right]_{t_1}^{t_2}\label{eq.EuLag-gen-L}\\
&+
\left.
\frac{\partial f(z_k^*,\dot{z}_k^*,\ddot{z}_k^*,t)}{\partial {\ddot z}_k}{\dot \eta}_k(t)
\right]_{t_1}^{t_2}
\nonumber\\ 
&+\int_{t_1}^{t_2}\left(
-\int_{t_1}^t\frac{\partial f(z_k^*,\dot{z}_k^*,\ddot{z}_k^*,\tau)}{\partial {\dot z}_k}d\tau
+\frac{\partial f(z_k^*,\dot{z}_k^*,\ddot{z}_k^*,t)}{\partial {\dot z}_k}
 \right.\nonumber\\
&-\left.\frac{d}{dt}\frac{\partial f(z_k^*,\dot{z}_k^*,\ddot{z}_k^*,t)}{\partial {\ddot z}_k}
\right){\dot \eta}_k(t)dt=0~~~k=1,..,K.\nonumber
\end{align}
In \eqref{eq.EuLag-gen-L} only the last integral remains, given that $\eta_k(t_2)={\dot \eta}_k(t_1)={\dot \eta}_k(t_2)=0$ and the integral $\int_{t_1}^t\frac{\partial f(z_k^*,\dot{z}_k^*,\ddot{z}_k^*,\tau)}{\partial {\dot z}_k}d\tau$ is zero at $t=t_1$ because the interval of integration has measure zero. Because  ${\dot \eta}_k(t)$ is entirely arbitrary, the optimum trajectory must satisfy \eqref{eq.EuLag-gen}. 
\end{proof}
Considering the definition of $f(z_k,\dot{z}_k,\ddot{z}_k,t)$ \eqref{def.f-gen}, the generalized Euler-Lagrange equation \eqref{eq.EuLag-gen} obtained in Lemma \ref{lem.gen-EuLag} implies that:
\begin{align}
&-\int_{t_1}^t\frac{\partial C_k(z_k^*,\dot{z}_k^*,\ddot{z}_k^*,\tau)}{\partial {\dot z}_k}d\tau
-\int_{t_1}^t(\overline{\beta}_k(\tau)
d\tau\\
&+\frac{\partial C_k(z_k^*,\dot{z}_k^*,\ddot{z}_k^*,t)}{\partial {\dot z}_k}
-\lambda(t)+\overline{\mu}_k(t)- \underline{\mu}_k(t)
\\
&
-\frac{d}{dt}\frac{\partial C_k(z_k^*,\dot{z}_k^*,\ddot{z}_k^*,t)}{\partial {\ddot z}_k}
-\frac{d\overline{\gamma}_k(t)}{dt}+\frac{d\underline{\gamma}_k(t)}{dt}=0.
\end{align}
To derive the price we need, once again, to focus on the marginal unit, for which the balance constraint is not tight. We also want to restrict ourselves to the assumption that the generators offering the bids which change their cost based on the energy
$C_k(z_k^*,\dot{z}_k^*,\ddot{z}_k^*,t)$ are only suppliers, which make sense since the power is constrained to be positive and if there was an import, it should be negative. This is a subtle point that leads to some simplifications that are not obvious in the general case in which suppliers can switch role and buy power to replenish their capacity. 

Once again, because of the KKT conditions and the fact that the multipliers are continuous for the marginal unit $
\overline{\mu}_k(t)=\underline{\mu}_k(t)=0$ and $d\overline{\gamma}_k(t)/dt=d\underline{\gamma}_k(t)/dt=0$. 
But what also is important to notice is that for the marginal unit also 
\begin{equation}
\int_{t_1}^t\overline{\beta}_k(\tau)d\tau=0
\end{equation}
 as soon as a unit saturates the capacity it can offer, 
i.e. $z_k=\overline{z}_k$, and $\overline{\beta}_k(\tau)>0$ the unit can no longer be the marginal unit as it has no more power to offer. 
Therefore, denoting by $k^*(t)$ the index of the marginal unit at time $t$ in this case we conclude:
\begin{theorem} 
When the generator costs depend on energy, power and ramp, i.e. $C_k(z_k,\dot{z}_k,\ddot{z}_k,t)$, the marginal price of power is:
\begin{align}\label{price.C(x,dotx,ddotx)}
\lambda^*(t)=&
\frac{\partial C_k(z_k^*,\dot{z}_k^*,\ddot{z}_k^*,t)}{\partial {\dot z}_k}-\frac{d}{dt}\frac{\partial C_k(z_k^*,\dot{z}_k^*,\ddot{z}_k^*,t)}{\partial {\ddot z}_k}
\nonumber\\
&
-\int_{t_1}^t\frac{\partial C_k(z_k^*,\dot{z}_k^*,\ddot{z}_k^*,\tau)}{\partial { z}_k}d\tau~~k=k^*(t)
\end{align}
\end{theorem}

A few remarks can be drawn out of this analysis which are summarized in the next section. 

\section{Discussion}

One of the first questions is how generating units would construct their bids. 
This question is hard to answer from a purely technical point of view, as the costs of ramps are mostly associated to the wear and tear of the generating units and, therefore, are not very easy to quantify, while the marginal cost of a certain schedule to offer energy for storage units can be close to zero. Hence, both costs are likely to be {\it opportunity driven costs} rather than clearly quantifiable {\it fuel costs}, as it is for the case of generation of power through fossil fuels. 
 However, the aim of our proposed bid is to 
allow the market participant to evaluate how to best position themselves to profit from the market environment
directly in the commitment phase. While day-ahead costs may rise, especially in the case of ramping services \cite{parvaniaHICSS2016}, the decisions result in lower real-time costs.  
The results we have derived in this paper, namely \eqref{eq.marginal} and \eqref{eq.mag.price} for a regular bid and \eqref{price.C(x,dotx)} and \eqref{price.C(x,dotx,ddotx)} for a generalized bid lead to the following observations
\begin{enumerate}
\item The cost of ramping is naturally apparent in the ED problem when the ramping constraints become tight through the derivative of the instantaneous Lagrange multiplier associated with the ramping constraints, even if one considers the current practice of bidding exclusively as a function of power.
\item Because the marginal price is tied to the marginal instantaneous cost, the price is continuous as long as the cost is continuous. Because net-demand is continuous it is natural to think that if demand lies in a finite function space (this was the hypothesis in \cite{parvania2015UC}) this function space might be adequate to also express the variations in costs. However, the connection between the features of the generators trajectories and the temporal variations in cost is still elusive and requires more investigation. 
\item The negative sign of the contributions that come from the dependence on ramping in \eqref{price.C(x,dotx)} has a nice interpretation: by increasing demand of power now, we can offset increasing costs of ramp that follow. If, on the other hand, the cost in ramp is declining (has negative slope in time) the rise in power is penalized.
\item For the energy cost the situation is different: inevitably as time goes on the energy demand grows which rises the marginal cost of power. Thus, the negative term \eqref{price.C(x,dotx,ddotx)}  can potentially be interpreted as a correction term, which rebalances such growth.  
\end{enumerate}

\section{Conclusions}
In this paper we derived the continuous time price of electricity in continuous time by solving an Economic dispatch problem with two differences compared to the standard formulation: 1) we included ramping constraints and solved the problem over a certain horizon accounting for possible ramping shortages; 2) we included more complex bids that allow for a change in price due to ramping or relative to the amount of energy that the unit has offered up to a certain point in time. 
We have shown that the notion of marginal price can be generalized to the continuous time regime and taken inspiration in solving the variational problem of computing the price from the solution of the so called isoperimentric problem \cite{variational}.

\section{Acknowledgements}
We wish to acknowledge the discussions with Masood Parvania and K\'{a}ri Hreinsson that lead to the question of correctly pricing power and energy trajectories. 

\bibliographystyle{IEEEtran}


\end{document}

%% file: macros.tex
\newcommand{\force}{\mbox{$\Vdash$}}

\newcommand{\ghat}{\mbox{$\bm \hat g$}}

\newcommand{\bzero}{\mbox{${\bm 0}$}}

\newcommand{\va}{\mbox{${\mathbf a}$}}
\newcommand{\vah}{\mbox{${\mathbf \hat a}$}}
\newcommand{\vat}{\mbox{${\mathbf \tilde a}$}}
\newcommand{\vb}{\mbox{${\mathbf b}$}}
\newcommand{\vd}{\mbox{${\mathbf d}$}}
\newcommand{\rh}{\mbox{${\hat r}$}}
\newcommand{\Itl}{\mbox{${\tilde I}$}}

\newcommand{\vxt}{\mbox{${\mathbf \tilde x}$}}
\newcommand{\vh}{\mbox{${\mathbf h}$}}
\newcommand{\vhh}{\mbox{${\mathbf \hat h}$}}
\newcommand{\ve}{\mbox{${\mathbf e}$}}
\newcommand{\vg}{\mbox{${\mathbf g}$}}
\newcommand{\vgh}{\mbox{${\mathbf \hat g}$}}
\newcommand{\vp}{\mbox{${\mathbf p}$}}
\newcommand{\vph}{\mbox{${\mathbf \hat p}$}}
\newcommand{\vq}{\mbox{${\mathbf q}$}}
\newcommand{\vt}{\mbox{${\mathbf t}$}}
\newcommand{\vw}{\mbox{${\mathbf w}$}}
\newcommand{\vwh}{\mbox{${\mathbf \hat w}$}}
\newcommand{\wh}{\mbox{${\hat w}$}}
\newcommand{\vwt}{\mbox{${\mathbf \tilde w}$}}
\newcommand{\wt}{\mbox{${\tilde w}$}}
\newcommand{\vs}{\mbox{${\mathbf s}$}}
\newcommand{\vsh}{\mbox{${\mathbf \hat s}$}}
\newcommand{\vst}{\mbox{${\mathbf \tilde s}$}}
\newcommand{\vr}{\mbox{${\mathbf r}$}}
\newcommand{\vx}{\mbox{${\mathbf x}$}}
\newcommand{\vv}{\mbox{${\mathbf v}$}}
\newcommand{\vu}{\mbox{${\mathbf u}$}}
\newcommand{\vy}{\mbox{${\mathbf y}$}}
\newcommand{\vz}{\mbox{${\mathbf z}$}}
\newcommand{\vn}{\mbox{${\mathbf n}$}}
\newcommand{\vnt}{\mbox{${\mathbf \tilde n}$}}
\newcommand{\vzero}{\mbox{${\mathbf 0}$}}
\newcommand{\vone}{\mbox{${\mathbf 1}$}}

\newcommand{\mA}{\mbox{{$\mathbf A$}}}
\newcommand{\mAh}{\mbox{${\mathbf \hat A}$}}
\newcommand{\mAt}{\mbox{$\mathbf \tilde A$}}
\newcommand{\mB}{\mbox{${\mathbf B}$}}
\newcommand{\mBh}{\mbox{${\mathbf \hat B}$}}
\newcommand{\mC}{\mbox{{$\mathbf C$}}}
\newcommand{\mCh}{\mbox{${\mathbf \hat C}$}}
\newcommand{\mD}{\mbox{{$\mathbf D$}}}
\newcommand{\mDt}{\mbox{$\mathbf \tilde D$}}
\newcommand{\mE}{\mbox{{$\mathbf E$}}}
\newcommand{\mG}{\mbox{{$\mathbf G$}}}
\newcommand{\mF}{\mbox{{$\mathbf F$}}}
\newcommand{\mH}{\mbox{{$\mathbf H$}}}
\newcommand{\mHb}{\mbox{${\mathbf \bar H}$}}
\newcommand{\mI}{\mbox{{$\mathbf I$}}}
\newcommand{\mIh}{\mbox{${\mathbf \hat I}$}}
\newcommand{\mN}{\mbox{{$\mathbf N$}}}
\newcommand{\mM}{\mbox{{$\mathbf M$}}}
\newcommand{\mMh}{\mbox{{$\mathbf \hat M$}}}
\newcommand{\mP}{\mbox{${\mathbf P}$}}
\newcommand{\mQ}{\mbox{${\mathbf Q}$}}
\newcommand{\mR}{\mbox{${\mathbf R}$}}
\newcommand{\mRh}{\mbox{${\mathbf {\hat {R}}}$}}
\newcommand{\mRt}{\mbox{${\mathbf \tilde R}$}}
\newcommand{\mS}{\mbox{${\mathbf S}$}}
\newcommand{\mSb}{\mbox{${\mathbf \bar S}$}}
\newcommand{\mSh}{\mbox{${\mathbf \hat S}$}}
\newcommand{\mSt}{\mbox{${\mathbf \tilde S}$}}
\newcommand{\mT}{\mbox{${\mathbf T}$}}
\newcommand{\mU}{\mbox{${\mathbf U}$}}
\newcommand{\mUh}{\mbox{${\mathbf \hat U}$}}
\newcommand{\mV}{\mbox{${\mathbf V}$}}
\newcommand{\mVh}{\mbox{${\mathbf \hat V}$}}
\newcommand{\mW}{\mbox{${\mathbf W}$}}
\newcommand{\mWh}{\mbox{${\mathbf \hat W}$}}
\newcommand{\mWt}{\mbox{${\mathbf \tilde W}$}}
\newcommand{\mX}{\mbox{${\mathbf X}$}}
\newcommand{\mY}{\mbox{${\mathbf Y}$}}
\newcommand{\mZ}{\mbox{${\mathbf Z}$}}


\newcommand{\ga}{\alpha}
\newcommand{\gb}{\beta}
\newcommand{\grg}{\gamma}
\newcommand{\gd}{\delta}
\newcommand{\gre}{\varepsilon}
\newcommand{\gep}{\epsilon}
\newcommand{\gz}{\zeta}
\newcommand{\gzh}{\mbox{$ \hat \zeta$}}
\newcommand{\gh}{\eta}
\newcommand{\gth}{\theta}
\newcommand{\gi}{iota}
\newcommand{\gk}{\kappa}
\newcommand{\gl}{\lambda}
\newcommand{\gm}{\mu}
\newcommand{\gn}{\nu}
\newcommand{\gx}{\xi}
\newcommand{\gp}{\pi}
\newcommand{\gph}{\phi}
\newcommand{\gr}{\rho}
\newcommand{\gs}{\sigma}
\newcommand{\gsh}{\hat \sigma}
\newcommand{\gt}{\tau}
\newcommand{\gu}{\upsilon}
\newcommand{\gf}{\varphi}
\newcommand{\gc}{\chi}
\newcommand{\go}{\omega}


\newcommand{\gG}{\Gamma}
\newcommand{\gD}{\Delta}
\newcommand{\gTh}{\Theta}
\newcommand{\gL}{\Lambda}
\newcommand{\gX}{\Xi}
\newcommand{\gP}{\Pi}
\newcommand{\gS}{\Sigma}
\newcommand{\gU}{\Upsilon}
\newcommand{\gF}{\Phi}
\newcommand{\gO}{\Omega}


\def\bm#1{\mbox{\boldmath $#1$}}
\newcommand{\vga}{\mbox{$\bm \alpha$}}
\newcommand{\vgb}{\mbox{$\bm \beta$}}
\newcommand{\vgd}{\mbox{$\bm \delta$}}
\newcommand{\vge}{\mbox{$\bm \epsilon$}}
\newcommand{\vgl}{\mbox{$\bm \lambda$}}
\newcommand{\vgm}{\mbox{$\bm \mu$}}
\newcommand{\vgr}{\mbox{$\bm \rho$}}
\newcommand{\vgn}{\mbox{$\bm \nu$}}
\newcommand{\vgp}{\mbox{$\bm \pi$}}
\newcommand{\vgrh}{\mbox{$\bm \hat \rho$}}
\newcommand{\vgrt}{\mbox{$\bm {\tilde \rho}$}}

\newcommand{\vgt}{\mbox{$\bm \gt$}}
\newcommand{\vgth}{\mbox{$\bm {\hat \tau}$}}
\newcommand{\vgtt}{\mbox{$\bm {\tilde \tau}$}}
\newcommand{\vpsi}{\mbox{$\bm \psi$}}
\newcommand{\vphi}{\mbox{$\bm \phi$}}
\newcommand{\vxi}{\mbox{$\bm \xi$}}
\newcommand{\vth}{\mbox{$\bm \theta$}}
\newcommand{\vthh}{\mbox{$\bm {\hat \theta}$}}

\newcommand{\mgG}{\mbox{$\bm \Gamma$}}
\newcommand{\mgGh}{\mbox{$\hat {\bm \Gamma}$}}
\newcommand{\mgD}{\mbox{$\bm \Delta$}}
\newcommand{\mgU}{\mbox{$\bm \Upsilon$}}
\newcommand{\mgL}{\mbox{$\bm \Lambda$}}
\newcommand{\mPsi}{\mbox{$\bm \Psi$}}
\newcommand{\mgX}{\mbox{$\bm \Xi$}}
\newcommand{\mgS}{\mbox{$\bm \Sigma$}}

\newcommand{\oA}{{\open A}}
\newcommand{\oC}{{\open C}}
\newcommand{\oF}{{\open F}}
\newcommand{\oN}{{\open N}}
\newcommand{\oP}{{\open P}}
\newcommand{\oQ}{{\open Q}}
\newcommand{\oR}{{\open R}}
\newcommand{\oZ}{{\open Z}}


\newcommand{\Nu}{{\cal V}}
\newcommand{\cA}{{\cal A}}
\newcommand{\cB}{{\cal B}}
\newcommand{\cC}{{\cal C}}
\newcommand{\cD}{{\cal D}}
\newcommand{\cF}{{\cal F}}
\newcommand{\cH}{{\cal H}}
\newcommand{\cK}{{\cal K}}
\newcommand{\cI}{{\cal I}}
\newcommand{\cL}{{\cal L}}
\newcommand{\cM}{{\cal M}}
\newcommand{\cN}{{\cal N}}
\newcommand{\cO}{{\cal O}}
\newcommand{\cP}{{\cal P}}
\newcommand{\cR}{{\cal R}}
\newcommand{\cS}{{\cal S}}
\newcommand{\cU}{{\cal U}}
\newcommand{\cV}{{\cal V}}
\newcommand{\cT}{{\cal T}}
\newcommand{\cX}{{\cal X}}

\newcommand{\rH}{^{*}}
\newcommand{\rT}{^{ \raisebox{1.2pt}{$\rm \scriptstyle T$}}}
\newcommand{\rF}{_{ \raisebox{-1pt}{$\rm \scriptstyle F$}}}
\newcommand{\rE}{{\rm E}}

\newcommand{\dom}{\hbox{dom}}
\newcommand{\rng}{\hbox{rng}}
\newcommand{\Span}{\hbox{span}}
\newcommand{\Ker}{\hbox{Ker}}
\newcommand{\On}{\hbox{On}}
\newcommand{\otp}{\hbox{otp}}
\newcommand{\ZFC}{\hbox{ZFC}}
\def\Re{\ensuremath{\hbox{Re}}}
\def\Im{\ensuremath{\hbox{Im}}}
\newcommand{\SNR}{\ensuremath{\hbox{SNR}}}
\newcommand{\CRB}{\ensuremath{\hbox{CRB}}}
\newcommand{\diag}{\ensuremath{\hbox{diag}}}
\newcommand{\trace}{\ensuremath{\hbox{tr}}}

\newcommand{\dlot}{\mbox{$\delta^1_3$}}
\newcommand{\Dlot}{\mbox{$\Delta^1_3$}}
\newcommand{\Dlof}{\mbox{$\Delta^1_4$}}
\newcommand{\dlof}{\mbox{$\delta^1_4$}}
\newcommand{\bP}{\mbox{$\mathbf{P}$}}
\newcommand{\Pot}{\mbox{$\Pi^1_2$}}
\newcommand{\Sot}{\mbox{$\Sigma^1_2$}}
\newcommand{\gDot}{\mbox{$\gD^1_2$}}

\newcommand{\Potr}{\mbox{$\Pi^1_3$}}
\newcommand{\Sotr}{\mbox{$\Sigma^1_3$}}
\newcommand{\gDotr}{\mbox{$\gD^1_3$}}

\newcommand{\Pofr}{\mbox{$\Pi^1_4$}}
\newcommand{\Sofr}{\mbox{$\Sigma^1_4$}}
\newcommand{\Dofr}{\mbox{$\gD^1_4$}}

\newcommand{\Sa}{\mbox{$S_{\ga}$}}
\newcommand{\Qk}{\mbox{$Q_{\gk}$}}
\newcommand{\Ca}{\mbox{$C_{\ga}$}}

\newcommand{\gkp}{\mbox{$\gk^+$}}
\newcommand{\aron}{ Aronszajn }

\newcommand{\sqkp}{\mbox{$\Box_{\gk}$}}
\newcommand{\dkp}{\mbox{$\Diamond_{\gk^{+}}$}}
\newcommand{\sqsqnce}
{\mbox{\\ $\ < \Ca \mid \ga < \gkp \ \ \wedge \ \ \lim \ga >$ \ \ }}
\newcommand{\dsqnce}{\mbox{$<S_{\ga} \mid \ga < \gkp >$}}



\newcommand{\beq}{\begin{equation}}
\newcommand{\eeq}{\end{equation}}
\newcommand{\bea}{\begin{array}}
\newcommand{\ena}{\end{array}}
\newcommand{\bds}{\begin {itemize}}
\newcommand{\eds}{\end {itemize}}
\newcommand{\bdf}{\begin{definition}}
\newcommand{\edf}{\end{definition}}
\newcommand{\blm}{\begin{lemma}}
\newcommand{\elm}{\end{lemma}}
\newcommand{\bthm}{\begin{theorem}}
\newcommand{\ethm}{\end{theorem}}
\newcommand{\bprp}{\begin{prop}}
\newcommand{\eprp}{\end{prop}}
\newcommand{\bcl}{\begin{claim}}
\newcommand{\ecl}{\end{claim}}
\newcommand{\bcr}{\begin{coro}}
\newcommand{\ecr}{\end{coro}}
\newcommand{\bquest}{\begin{question}}
\newcommand{\equest}{\end{question}}

\newcommand{\rarrow}{{\rightarrow}}
\newcommand{\Rarrow}{{\Rightarrow}}
\newcommand{\larrow}{{\larrow}}
\newcommand{\Larrow}{{\Leftarrow}}
\newcommand{\restrict}{{\upharpoonright}}
\newcommand{\nin}{{\not \in}}



\newcommand{\ie}{\hbox{i.e.}}
\newcommand{\eg}{\hbox{e.g.}}